\documentclass{cimart}

\usepackage{enumerate}
\usepackage{mathtools}

\newcommand{\vertbar}{\>|\>}
\newcommand{\set}[2]{\ensuremath{\{ #1 \vertbar #2 \}}}
\DeclareMathOperator{\Der}{Der}
\DeclareMathOperator{\id}{id}
\DeclareMathOperator{\Sl}{\mathsf{sl}}

\title{
A $\delta$-first Whitehead Lemma for Jordan algebras
    }

\author{
Arezoo Zohrabi and Pasha Zusmanovich
    }

\authorinfo[
A. Zohrabi]{
University of Ostrava, Czech Republic}{%
arezoo.zohrabi@osu.cz
    }

\authorinfo[
P. Zusmanovich]{
University of Ostrava, Czech Republic}{%
pasha.zusmanovich@osu.cz
    }

\abstract{%
We compute $\delta$-derivations of simple Jordan algebras with values in 
irreducible bimodules. They turn out to be either ordinary derivations 
($\delta = 1$), or scalar multiples of the identity map ($\delta = \frac 12$).
This can be considered as a generalization of the ``First Whitehead Lemma'' for
Jordan algebras which claims that all such ordinary derivations are inner. The 
proof amounts to simple calculations in matrix algebras, or, in the case of 
Jordan algebras of a symmetric bilinear form, to more elaborated calculations in
Clifford algebras.
    }



\VOLUME{33}
\YEAR{2025}
\ISSUE{1}
\NUMBER{2}
\DOI{https://doi.org/10.46298/cm.13595}

\begin{document}

\section*{Introduction}

Let $A$ be a (generally, nonassociative) algebra, and $M$ an $A$-bimodule, with
the bimodule action denoted by $\bullet$, and $\delta$ an element of the ground
field. Recall that a \emph{$\delta$-derivation} of $A$ with values in $M$ is a 
linear map $D: A \to M$ such that
\begin{equation}\label{eq-delta}
D(xy) = \delta D(x) \bullet y + \delta x \bullet D(y)
\end{equation}
for any $x,y \in A$. Obviously, the ordinary derivations are $1$-derivations, 
and in the case where $M = A$, the regular bimodule, elements of the centroid 
are the special cases of $\frac 12$-derivations. Despite being seemingly a 
straightforward generalization of derivations, $\delta$-derivations appear in
different situations and are proved to be a useful and interesting invariant. 
For example, in Lie algebras context, they are related, for various values of
$\delta$, to commutative $2$-cocycles, arise in description of (ordinary) 
derivations of certain current Lie algebras, and can be used to construct 
non-semigroup gradings; see \cite{delta} for further references.

There are quite a lot of 
investigations of  $\delta$-derivations of various algebras with values in 
itself, i.e., in the regular module --- see, for example, in addition to the
already mentioned \cite{delta}, also \cite{kayg-first} and \cite{octonion}, and 
references therein --- but very little has been done concerning 
$\delta$-derivations with values in more general modules. 

In \cite{delta-whitehead} we computed $\delta$-derivations of simple 
finite-dimensional Lie algebras of characteristic zero with coefficients in 
finite-dimensional modules. The motivation of doing this was twofold: first, to 
establish a ``$\delta$-analog'' of the classical First Whitehead Lemma, that is,
to prove that $\delta$-derivations of a finite-dimensional simple Lie algebra of
characteristic zero with values in a finite-dimensional module are just inner derivations, with the exception of peculiar cases related to $\Sl(2)$. Second, 
to provide an alternative route to computation, done in \cite{octonion}, of 
$\delta$-derivations of algebras of skew-Hermitian matrices over octonions, an 
interesting series of anticommutative nonassociative algebras.

The purpose of this note is to establish a ``Jordan analog'' of this result: we
prove that any $\delta$-derivation of a finite-dimensional simple Jordan algebra
with values in a finite-dimensional unital irreducible bimodule is, in a sense, 
trivial, i.e., it is either an ordinary, and hence inner, derivation 
($\delta = 1)$, or is a scalar multiple of the identity map on the underlying 
Jordan algebra, in which case the bimodule is the regular bimodule 
($\delta = \frac 12$). (Note that, unlike in the Lie algebras case, there are no
exceptional cases related to Jordan algebras of small dimension). This can be 
considered as a generalization of the ``First Whitehead Lemma'' for Jordan 
algebras, that is, the classical result that (ordinary) derivations of a 
finite-dimensional simple Jordan algebra with values in an irreducible bimodule
are inner. 

As a corollary, we also present an alternative proof of the result, established
in \cite{octonion}, that $\delta$-derivations of the algebra of Hermitian 
$n \times n$ matrices over octonions are trivial (which, in its turn, helped to
determine symmetric invariant bilinear forms on these algebras). These 
octonionic matrix algebras generalize the $3 \times 3$ case of the 
$27$-dimensional exceptional simple Jordan algebra, and the $4 \times 4$ case 
appears in modern physical theories. (Note, however, that for $n \ge 4$ these
algebras are neither Jordan, nor belong to any known variety of nonassociative
algebras studied in the literature). The proof uses the fact that the 
$n \times n$ octonionic matrix algebra contains the simple Jordan subalgebra 
$M_n^+(K)$ of symmetric matrices, and restricting a $\delta$-derivation to that
particular subalgebra gives a $\delta$-derivation of $M_n^+(K)$ with values in 
the whole octonionic matrix algebra.

\section*{Notation and conventions}

Unless stated otherwise, the ground field $K$ is assumed to be arbitrary of
characteristic different from $2$, a usual assumption in the (classical) Jordan
structure theory. As we deal simultaneously with Jordan algebras and associative
algebras (as their associative envelopes), the multiplication in the former 
(or, more generally, in an arbitrary commutative nonassociative algebra, like 
in Lemma \ref{lemma-1} below) will be (traditionally) denoted by $\circ$, while
multiplication in the latter will be denoted by juxtaposition. Action of a 
Jordan algebra on its bimodule is denoted by $\bullet$. The action of an algebra
on itself by multiplication is called the \emph{regular} representation.

For an associative algebra $A$, $A^{op}$ denotes the ``opposite'' algebra, i.e.,
the same vector space $A$ subject to multiplication $x \cdot y = yx$, and 
$A^{(+)}$ denotes its ``plus'' Jordan algebra, i.e., the same vector space $A$ 
subject to multiplication $x \circ y = \frac 12 (xy + yx)$. Assuming an algebra
$A$ has an involution (i.e., an antiautomorphism of order $2$) 
$a \mapsto a^{\mathbb J}$, define 
$S^{+}(A,\mathbb J) = \set{a\in A}{a^{\mathbb J} = a}$ and 
$S^{-}(A,\mathbb J) = \set{a\in A}{a^{\mathbb J} = -a}$, the vector spaces of 
\mbox{$\mathbb{J}$-symmetric} and $\mathbb J$-skew-symmetric elements, respectively. If
$A$ is associative, then $S^{+}(A,\mathbb J)$ is a Jordan subalgebra of 
$A^{(+)}$. $M_n(A)$ denotes the matrix algebra of degree $n$ over the algebra 
$A$.

\section{Recapitulation on Clifford algebras}\label{sec-clifford}

Here we recall some facts related to Clifford algebras which will be used below,
when dealing with Jordan algebras of a symmetric bilinear form.

Let $V$ be a finite-dimensional vector space equipped with a nondegenerate 
symmetric bilinear form $f$, and $C(V,f)$, or just $C(V)$ if there is no 
ambiguity what $f$ is, a corresponding Clifford algebra.

If $W$ is a subspace of $V$, then, denoting by abuse of notation the restriction
of $f$ to $W$ by the same letter $f$, we have that $C(W,f)$ is a subalgebra of
$C(V,f)$.

Fix an orthogonal, with respect to $f$, basis $\{u_1,\dots,u_{n}\}$ of $V$. Then
the basis of $C(V)$ can be chosen to consist of elements 
\begin{equation}\label{eq-u}
u_{i_1} u_{i_2} \dots u_{i_k}, 
\text{ where } 1 \le i_1 < i_2 < \dots < i_k \le n,\> 0 \le k \le n
\end{equation}
(as usual in such situations, we tacitly assume that the above product for 
$k=0$, i.e., with the zero number of factors, is equal to $1$). 

Let $C(V)^{(k)}$ be the $\binom{n}{k}$-dimensional subspace of $C(V)$ spanned 
by elements of the form (\ref{eq-u}) for a fixed $k$. (Thus, $C(V)^{(0)}$ is 
just the one-dimensional space $K1$, and $C(V)^{(1)} = V$). Then $C(V)$ is 
decomposed as the vector space direct sum $\bigoplus_{k=0}^{n} C(V)^{(k)}$. 

\begin{lemma}
In the Clifford algebra $C(V,f)$ the following equality holds for any 
$x,y_1,\dots,y_k \in V$:
\begin{equation}\label{eq-prod}
y_1 \cdots y_k x = (-1)^k x y_1 \cdots y_k + 
2\sum_{i=1}^k 
(-1)^{k+i} f(x,y_i) y_1 \cdots y_{i-1} \widehat{y_i} y_{i+1} \cdots y_k
\end{equation}
(as usual, $\>\>\>\>\widehat{}\>\>\>\>$ means that the corresponding element is
omitted in the product). 
\end{lemma}

\begin{proof}
This is implicit in the corresponding Clifford-algebraic calculations in 
\cite{osaka} and \cite[Chapter VII, \S 1]{jacobson}. Since we failed to find an
explicit proof in the literature, we provide the proof here.

We perform induction on $k$. For $k=1$ the equality (\ref{eq-prod}) reduces to 
$xy_1 + y_1x = 2f(x,y_1)$, the defining relation in the Clifford algebra. Now 
assume (\ref{eq-prod}) holds for a certain $k$. Then
\begin{align*}
y_1 \cdots y_k y_{k+1} x
={ }&{ } y_1 \cdots y_k \big(-x y_{k+1} + 2f(x,y_{k+1})\big) \\
={ }&{ } - y_1 \cdots y_k x y_{k+1} + 2f(x,y_{k+1}) y_1 \cdots y_k \\
={ }&{ } -\Big((-1)^k x y_1 \cdots y_k + 2\sum_{i=1}^k 
(-1)^{k+i} f(x,y_i) y_1 \cdots y_{i-1} \widehat{y_i} y_{i+1} \cdots y_k\Big) 
y_{k+1} \\
    &+ 2f(x,y_{k+1}) y_1 \cdots y_k \\
={ }&{ } (-1)^{k+1} x y_1 \cdots y_{k+1} + 2\sum_{i=1}^{k+1} 
(-1)^{k+1+i} f(x,y_i) y_1 \cdots y_{i-1} \widehat{y_i} y_{i+1} \cdots y_{k+1} ,
\end{align*}
as desired.
\end{proof}

\begin{corollary}\label{cor}
\begin{equation*}
u_i(u_1 \cdots u_k)u_i = \begin{cases}
(-1)^{k+i} f(u_i,u_i) u_1 \cdots u_k, \text{ if } i \in \{1,\dots,k\}
\\
(-1)^k \>\>\>\>\>\> 
f(u_i,u_i) u_1 \cdots u_k, \text{ if } i \notin \{1,\dots,k\} .
\end{cases}
\end{equation*}
\end{corollary}

\begin{proof}
Multiplying both sides of (\ref{eq-prod}) with $x$ from the left, and using the 
fact that \mbox{$x^2 = f(x,x)1$}, we get
\begin{equation*}
x y_1 \cdots y_k x = (-1)^k f(x,x) y_1 \cdots y_k + 
2\sum_{i=1}^k 
(-1)^{k+i} f(x,y_i) x y_1 \cdots y_{i-1} \widehat{y_i} y_{i+1} \cdots y_k .
\end{equation*}

Substituting in the last equality the appropriate $u_i$'s instead of $x$ and
$y_i$'s, and using the fact that $u_i$ and $u_j$ anticommute for $i \ne j$, we 
get the required equalities.
\end{proof}

\section[Recapitulation on Jordan algebras and bimodules]{Recapitulation on simple Jordan algebras and their irreducible 
bimodules}\label{sec-rev}

Here we briefly review the necessary facts about Jordan algebras and their 
modules. The main sources are \cite{osaka} and 
\cite{jacobson}, Chapters II, V, and VII.

If $J$ is a simple Jordan algebra, the associative universal envelope $U(J)$ 
(called universal associative algebra for the unital representations in 
\cite{osaka}, and universal unital multiplication envelope in \cite{jacobson}) 
is semisimple, and unital irreducible $J$-bimodules are in a bijective 
correspondence with simple components of $U(J)$ (note, however, that since not 
every Jordan algebra is special, $J$ is not necessarily embedded into 
$U(J)^{(+)}$).

Let us recall the isomorphism types of finite-dimensional simple Jordan 
algebras, their associative universal envelopes, and their irreducible 
bimodules.

\begin{enumerate}[\upshape(i)]
\setlength{\itemsep}{9pt}

\item\label{it-1}
The ground field $K$. The associative universal envelope coincides with $K$, and
any unital irreducible $K$-bimodule is isomorphic to $K$ itself.

\item 
The algebra $J(V,f) = K1 \oplus V$ of nondegenerate symmetric bilinear form $f$
defined on a vector space $V$ of dimension $\ge 2$. The multiplication between 
elements $x,y\in V$ is defined by $x \circ y = f(x,y)1$. Define the vector space
$V^{ev}$ as follows:
\begin{enumerate}[\upshape(a)]
\item
If $\dim V$ is even, set $V^{ev} = V$;
\item
If $\dim V$ is odd, set $V^{ev} = V \oplus Ku$, and extend $f$ to $V^{ev}$ by 
setting $f(u,u) = 1$ and $f(u,V) = f(V,u) = 0$.
\end{enumerate}
The associative universal envelope is isomorphic to the so-called meson algebra;
this is the unital subalgebra of the algebra of linear endomorphisms of the 
Clifford algebra $C(V^{ev},f)$, generated by Jordan multiplications on the 
elements of $J(V,f)$, i.e., by the maps $x \mapsto \frac12 (ax + xa)$, where 
$a \in J(V,f)$. The Jordan algebra $J(V,f)$ is embedded into the algebra 
$C(V^{ev},f)^{(+)}$, and the action of $J(V,f)$ on $C(V^{ev},f)$ is the 
restriction to $J(V,f)$ of the regular representation of $C(V^{ev},f)^{(+)}$. Below we describe the decomposition of $C(V^{ev},f)$ into 
the direct sum of subspaces invariant with respect to the action of $J(V,f)$. The 
unital irreducible $J(V,f)$-bimodules are in one-to-one correspondence with these subspaces. The precise description, in terms of a fixed orthogonal basis
of $V$, depends on the parity (more exactly, on the residue modulo $4$) of $n$.

\smallskip
\noindent
Case (a):
$n = 2m$. The irreducible $J(V,f)$-invariant subspaces are:
\begin{equation*}
C(V)^{(0)} \oplus C(V)^{(1)},\quad C(V)^{(2)} \oplus C(V)^{(3)},
\quad \dots,\quad C(V)^{(2m-2)} \oplus C(V)^{(2m-1)},\quad C(V)^{(2m)} .
\end{equation*}

\smallskip
\noindent
Case (b):
$n = 2m-1$. Extend the chosen orthogonal basis of $V$ to an orthogonal basis
of $V^{ev}$ by adding $u$ to it.

\smallskip
\noindent
Case (b1): $m$ is even. The irreducible $J(V,f)$-invariant subspaces are:
\begin{align*}
&C(V)^{(0)} \oplus C(V)^{(1)},\quad C(V)^{(2)} \oplus C(V)^{(3)},
\quad \dots,\quad C(V)^{(m-2)} \oplus C(V)^{(m-1)}, 
\\
&C(V)^{(0)}u,\quad C(V)^{(1)}u \oplus C(V)^{(2)}u,
\quad \dots,\quad C(V)^{(m-3)}u \oplus C(V)^{(m-2)}u .
\end{align*}
In addition to that, the subspace $C(V)^{(m-1)}u \oplus C(V)^{(m)}u$ decomposes
as the direct sum of two irreducible $J(V,f)$-invariant subspaces of the same 
dimension $\frac{1}{2}\binom{2m}{m}$.

\smallskip
\noindent
Case (b2): $m$ is odd. The irreducible $J(V,f)$-invariant subspaces are:
\begin{align*}
&C(V)^{(0)} \oplus C(V)^{(1)},\quad C(V)^{(2)} \oplus C(V)^{(3)},
\quad \dots,\quad
C(V)^{(m-3)} \oplus C(V)^{(m-2)}, 
\\
&C(V)^{(0)}u,\quad C(V)^{(1)}u \oplus C(V)^{(2)}u,
\quad \dots,\quad C(V)^{(m-2)}u \oplus C(V)^{(m-1)}u .
\end{align*}
In addition to that, the subspace $C(V)^{(m-1)} \oplus C(V)^{(m)}$ 
decomposes as the direct sum of two irreducible $J(V,f)$-invariant subspaces of
the same dimension $\frac{1}{2}\binom{2m}{m}$.

\item 
The algebra $A^{(+)}$, where $A$ is a central simple associative algebra. The 
associative universal envelope is $A \oplus A^{\mathbb J}$, where $\mathbb J$ is
an involution interchanging $A$ and $A^{\mathbb J}$. The unital irreducible 
$A^{(+)}$-bimodules are as follows. First, $A$ (the regular bimodule), with the 
action $a \bullet b = \frac 12 (ab + ba)$ (here and below, we assume $a$ belongs
to the Jordan algebra in question, and $b$ to the corresponding bimodule); then,
assuming $A$ has an involution $\mathbb K$, there are additionally the following
four bimodules:
\begin{enumerate}
\item $S^{+}(A,\mathbb K)$, $a \bullet b = \frac 12 (ab + ba^{\mathbb K})$;
\item $S^{+}(A,\mathbb K)$, $a \bullet b = \frac 12 (a^{\mathbb K}b + ba)$;
\item $S^{-}(A,\mathbb K)$, $a \bullet b = \frac 12 (ab + ba^{\mathbb K})$;
\item $S^{-}(A,\mathbb K)$, $a \bullet b = \frac 12 (a^{\mathbb K}b + ba)$.
\end{enumerate}

\item 
$S^{+}(A,\mathbb J)$, where $A$ is a central simple associative algebra with an
involution $\mathbb J$ of the first kind. The associative universal envelope is
$A$, and the unital irreducible $S^{+}(A,\mathbb J)$-bimodules are 
$S^{+}(A,\mathbb J)$ (the regular bimodule) and $S^{-}(A,\mathbb J)$, both with
the action $a \bullet b = \frac 12 (ab + ba)$.

\item\label{it-5} 
$S^{+}(A,\mathbb J)$, where $A$ is a simple associative algebra whose center 
$Z(A)$ is a quadratic extension of the base field $K$, $\mathbb J$ is an 
involution of the second kind on $A$, and $A$ over $Z(A)$ has an involution 
$\mathbb K$ commuting with $\mathbb J$. The associative universal envelope is 
$A$, and the unital irreducible $S^{+}(A,\mathbb J)$-bimodules are:
\begin{enumerate}
\item $S^{+}(A,\mathbb J)$, $a \bullet b = \frac 12 (ab + ba)$ 
(the regular bimodule);
\item $S^{+}(A,\mathbb K)$, $a \bullet b = \frac 12 (ab + ba^{\mathbb K})$;
\item $S^{-}(A,\mathbb K)$, $a \bullet b = \frac 12 (ab + ba^{\mathbb K})$.
\end{enumerate}

\item\label{it-6} 
The $27$-dimensional algebra of $3\times 3$ Hermitian matrices over octonions.
The associative universal envelope is isomorphic to the full $27 \times 27$ 
matrix algebra $M_{27}(K)$, and the only unital irreducible bimodule is the 
regular one.

\end{enumerate}

\smallskip

The algebras specified in (\ref{it-1})--(\ref{it-5}) are special, while the last
one, in (\ref{it-6}), is exceptional.

To summarize the special cases: any unital irreducible bimodule of a special 
simple Jordan algebra $J$ is a submodule of the restriction to $J$ either of the
regular representation of $U(J)^{(+)}$, or of the representation of $U(J)^{(+)}$
in itself with one of the following actions: 
\begin{equation}\label{eq-b1}
a \bullet b = \frac 12 (ab + ba^{\mathbb K})
\end{equation}
or 
\begin{equation}\label{eq-b2}
a \bullet b = \frac 12 (a^{\mathbb K}b + ba) ,
\end{equation}
where $a,b \in U(J)^{(+)}$, and $\mathbb K$ is an involution on $U(J)$. Note, 
however, that for an arbitrary associative algebra $A$ with an involution
$\mathbb K$, the action of type (\ref{eq-b2}) coincides with the action of type
(\ref{eq-b1}) for the algebra $A^{op}$ (with the same involution $\mathbb K$).

\section{\texorpdfstring{$\delta$-derivations, simple case}{delta-derivations, simple case}}\label{sec-delta}

\begin{lemma}\label{lemma-1}
Let $D$ be a nonzero $\delta$-derivation of a commutative algebra $A$ with unit 
with values in a unital symmetric $A$-bimodule $M$. Then either $\delta = 1$
(i.e., $D$ is a derivation), or $\delta=\frac 12$ and $D(x) = x \bullet m$ for 
a certain element $m\in M$ such that
$$
(x \circ y) \bullet m = 
\frac 12\big(x \bullet (y \bullet m) + y \bullet (x \bullet m)\big)
$$
for any $x,y \in A$.
\end{lemma}

\begin{proof}
An elementary proof consists of a repetitive substitution of $1$ in the equation
(\ref{eq-delta}), and is an almost verbatim repetition of the proof of Lemma 9 
from \cite{octonion} which treats the case $M = A$ (which, in its turn, is just
a slight reformulation of Theorem 2.1 from \cite{kayg-first}).
\end{proof}

\begin{lemma}\label{lemma-2}
Let $J$ be a special Jordan algebra with unit, embedded into an algebra 
$A^{(+)}$ for an associative algebra $A$. Let $M$ be a submodule of the 
restriction to $J$ of the regular $A^{(+)}$-bimodule, and $D: J \to M$ a nonzero
$\delta$-derivation. Then either $\delta = 1$, or $\delta = \frac 12$ and 
$D(x) = \frac 12 (xm + mx)$ for a certain element $m \in M$ such that
\begin{equation}\label{eq-2}
xym + yxm + mxy + myx - 2xmy - 2ymx = 0
\end{equation}
for any $x,y \in J$.
\end{lemma}

\begin{lemma}\label{lemma-3}
Let $J$ be a special Jordan algebra with unit, embedded into an algebra 
$A^{(+)}$ for an associative algebra $A$ with involution $\mathbb K$. Let $M$ be
a submodule of the restriction to $J$ of the $A^{(+)}$-bimodule with the action
$a \bullet b = \frac 12 (ab + ba^{\mathbb K})$, and $D: J \to M$ a nonzero
$\delta$-derivation. Then either $\delta = 1$, or $\delta = \frac 12$ and 
$D(x) = \frac 12 (xm + mx^{\mathbb K})$ for a certain element $m \in M$ such 
that
\begin{equation}\label{eq-K}
  xym + yxm + mx^{\mathbb K}y^{\mathbb K} + my^{\mathbb K}x^{\mathbb K} 
- 2xmy^{\mathbb K} - 2ymx^{\mathbb K} = 0
\end{equation}
for any $x,y \in J$.
\end{lemma}

\begin{proof}
The proof of Lemma \ref{lemma-2} (resp. Lemma \ref{lemma-3}), amounts to 
application of Lemma~\ref{lemma-1} to the situation where 
$x \circ y = \frac 12 (xy + yx)$ and $x \bullet m = \frac 12 (xm + mx)$ (resp. $x \bullet m = \frac 12 (xm + mx^{\mathbb K})$).
\end{proof}

In what follows it will be useful the rewrite the condition (\ref{eq-2}) in
equivalent forms. Substituting $y=x$ in (\ref{eq-2}), we get the equality 
\begin{equation}\label{eq-xmx}
x^2 m + mx^2 - 2xmx = 0 ,
\end{equation}
valid for any $x \in J$ (actually, this is equivalent to (\ref{eq-2}) via 
linearization). The last equality, in its turn, can be rewritten as
\begin{equation}\label{eq-comm}
[[m,x],x] = 0 ,
\end{equation}
where $[a,b] = ab - ba$ is the commutator of two elements.

Similarly, (\ref{eq-K}) can be written in an equivalent form by setting $y=x$:
\begin{equation}\label{eq-K2}
x^2 m + m (x^2)^{\mathbb K} - 2xmx^{\mathbb K} = 0 ,
\end{equation}
valid for any $x\in J$.

\begin{theorem}\label{th-1}
Let $D$ be a nonzero $\delta$-derivation of a finite-dimensional simple Jordan 
algebra with values in a finite-dimensional unital irreducible bimodule $M$. 
Then either $\delta=1$ and $D$ is an inner derivation, or $\delta=\frac 12$, $M$
is the regular bimodule, and $D$ is a scalar multiple of the identity map.
\end{theorem}

\begin{proof}
If $\delta = 1$ then $D$ is an ordinary derivation and hence is inner, so
suppose $\delta \ne 1$. We proceed case-by-case according to \S \ref{sec-rev}.
Note that since our problem -- computation of $\delta$-derivations -- does 
not change under an extension of the ground field (see the remark after 
Theorem~8 in \cite[\S 3]{octonion}), we may assume that the ground field $K$ is
algebraically closed whenever appropriate.

\medskip

(i) Obvious.

\medskip

(ii) As $x^2 = f(x,x)1$ for any $x\in V$, the equality (\ref{eq-xmx}) reduces to
\begin{equation}\label{eq-xmx2}
xmx = f(x,x)m
\end{equation}
for any $x\in V$. 

We have that either $m \in C(V)^{(k)} \oplus C(V)^{(k+1)}$, or, in the case of 
odd-dimensional $V$, $m \in C(V)^{(k)}u \oplus C(V)^{(k+1)}u$ for certain $k$.
Decompose $m$ as a linear combination of the basic elements of the form
(\ref{eq-u}), and let $u_{i_1} \cdots u_{i_k}$, or $u_{i_1} \cdots u_{i_k} u$ in
the case of odd-dimensional $V$, be one of those basic elements entering this 
sum with a nonzero coefficient. Assuming $k \ge 2$, Corollary \ref{cor} and 
equality (\ref{eq-xmx2}) for $x = u_{i_1}$ coupled together, yield
$$
f(u_{i_1},u_{i_1}) u_{i_1} \cdots u_{i_k} = 
u_{i_1} (u_{i_1} \cdots u_{i_k}) u_{i_1} = 
(-1)^{k+1} f(u_{i_1},u_{i_1}) u_{i_1} \cdots u_{i_k} 
$$
and
$$
f(u_{i_2},u_{i_2}) u_{i_1} \cdots u_{i_k} = 
u_{i_2} (u_{i_1} \cdots u_{i_k}) u_{i_2} = 
(-1)^{k+2} f(u_{i_2},u_{i_2}) u_{i_1} \cdots u_{i_k} ,
$$
whence
$$
u_{i_1} \cdots u_{i_k} = 
(-1)^{k+1} u_{i_1} \cdots u_{i_k} = (-1)^{k+2} u_{i_1} \cdots u_{i_k} = 0 ,
$$
a contradiction. The same reasoning holds when applied to the element of the 
form $u_{i_1} \cdots u_{i_k} u$. 

Hence, in both cases, $k \le 1$, i.e., $m$ belongs either to the regular 
bimodule $K \oplus V$, or to the module $Ku$, or $m \in Vu$. These remaining 
cases can be processed in the similar way as the generic case above, just the 
calculations are even simpler. Namely, in the case of the regular module we can
either refer to \cite[Lemma 2.3]{kayg-first}, or, assuming $u_i$ is a nonzero component of 
$m$, consider the expression $u_j u_i u_j$ with $j \ne i$, and apply 
Corollary~\ref{cor} again to get a contradiction. A similar reasoning works in 
the case $m = \lambda u$ for some $\lambda \in K$, and in the remaining case 
$m \in Vu$, assuming $u_i u$ is a nonzero component of $m$, we get a 
contradiction in a similar fashion by considering expression $u_i(u_i u)u_i$. 
Therefore, the only remaining possibility is $m \in K1$.

\medskip

(iii)
For the case of the regular $A^{(+)}$-bimodule we can either refer to 
\cite[Theorem 2.5]{kayg-first}, or proceed as follows. Assume that the ground 
field $K$ is algebraically closed. Then $A$ is isomorphic to the full matrix 
algebra over $K$, and substituting in (\ref{eq-comm}) the matrix units $E_{ii}$
instead of $x$, we get that $m$ should be a diagonal matrix, with diagonal 
elements, say, $\lambda_1, \dots, \lambda_n$. Then, for example, substituting in
(\ref{eq-comm}) again the matrix having $1$'s in the first row and the first 
column, and zeros elsewhere, we get at the left-hand side the matrix whose first
row, beginning from the 2nd element, consists of elements 
$\lambda_1 - \lambda_i$. This proves that all $\lambda_i$'s are equal, and $m$ 
is a scalar multiple of the identity matrix.

Very similar reasonings will do in the case of bimodule $S^+(A,\mathbb K)$ or 
$S^-(A,\mathbb K)$ (i.e., the space of symmetric or skew-symmetric matrices in 
the case of algebraically closed ground field) with the action 
$a \bullet b = \frac 12 (ab + ba^{\mathbb K})$, utilizing the equality 
(\ref{eq-K2}). In these cases it is enough to substitute there the matrix units
$E_{ij}$ instead of $x$, to get $m=0$.

The remaining two cases can be reduced to the already considered cases by 
passing to the algebra $A^{op}$ with the opposite multiplication (which does not
change the associated Jordan algebra $A^{(+)}$), as noted at the end of 
\S \ref{sec-rev}.

\medskip

(iv) Over an algebraically closed field the Jordan algebra $S^+(A,\mathbb J)$ is
isomorphic either to the algebra of $n \times n$ symmetric matrices, or to the 
algebra of $2n \times 2n$ matrices of the form
\begin{equation}\label{eq-mat}
\left(\begin{matrix}
A & B \\
C & A^\top
\end{matrix}\right)
\end{equation}
where $A$ is an arbitrary $n \times n$ matrix, and $B$ and $C$ are 
skew-symmetric $n \times n$ matrices.

We are again in the realm of Lemma~\ref{lemma-2}, and in the first case 
(symmetric matrices) exactly the same reasonings as in case (iii) will do. 
Indeed, when deriving from (\ref{eq-comm}) the necessary conclusion about $m$ by
substituting various matrices instead of $x$, we used only symmetric matrices. 
Thus the same conclusion holds in this case, i.e., $m$ is a scalar multiple of the identity matrix.

In the second case (matrices of type (\ref{eq-mat})) we proceed by ``doubling'' 
the reasoning in the previous case. Namely, substituting in (\ref{eq-comm}) the
matrices
\begin{equation*}
\left(\begin{matrix}
E_{ii} & 0 \\
0 & E_{ii}
\end{matrix}\right)
\end{equation*}
instead of $x$, we get that $m$ should be a diagonal matrix, and then 
substituting in (\ref{eq-comm}) the matrix
\begin{equation*}
\left(\begin{matrix}
M & 0 \\
0 & M
\end{matrix}\right)
\end{equation*}
where $M$ has $1$'s in the first row and the first column and zeros elsewhere 
(so $M^\top = M$), we get that $m$ is a scalar multiple of the identity matrix.

Thus, in both cases $m$ is a scalar multiple of the identity matrix, and this is
possible only in the case when the bimodule coincides with $S^+(A,\mathbb J)$, 
i.e., is the regular bimodule.

\medskip

(v) Being extended to an algebraic closure of the ground field, these algebras 
and bimodules are isomorphic to those of the case (iii).

\medskip

(vi) As in this case the question is reduced entirely to $\delta$-derivations of
simple exceptional Jordan algebra with values in itself, we can refer to 
\cite[Theorem 2.5]{kayg-first}, or, to a more general case of Hermitian matrices
over octonions of arbitrary size covered in \cite[Theorem 8]{octonion}. (Note
that below, in \S \ref{sec-octonion}, we outline a proof, different from those
given in \cite{octonion}, of the triviality of $\delta$-derivations of these 
algebras using Theorem \ref{th-1}. This does not lead to circular arguments, 
as in \S \ref{sec-octonion} we are using the part of Theorem \ref{th-1} dealing
with special matrix Jordan algebras).

\medskip

To summarize: in all the cases we have proved that in the context of 
Lemma~\ref{lemma-1}, $M$ is the regular bimodule and $m$ is a scalar multiple of
$1$, and hence $D$ is a scalar multiple of the identity map on the underlying 
Jordan algebra.
\end{proof}

\section{\texorpdfstring{$\delta$-derivations, semisimple case}{delta-derivations, semisimple case}}\label{sec-semi}

As an immediate corollary of the just proved theorem, we get a description of
$\delta$-derivations of a finite-dimensional semisimple Jordan algebra $J$ with 
values in an arbitrary finite-dimensional unital bimodule $M$. For that, we need
first the following two simple lemmas, valid for an arbitrary nonassociative
algebra $A$. By $\Der_\delta(A,M)$ we denote the vector space of all 
$\delta$-derivations of an algebra $A$ with values in an $A$-bimodule $M$.

\begin{lemma}\label{lemma-5}
Suppose $M$ is decomposed as the direct sum of $A$-submodules: 
$M = \bigoplus_i M_i$. Then
\begin{equation}\label{eq-sum}
\Der_\delta(A,M) \simeq \bigoplus_i \Der_\delta(A,M_i) .
\end{equation}
\end{lemma}

\begin{proof}
The elementary proof of the Lie algebra case given in 
\cite[Lemma 1]{delta-whitehead} is valid for arbitrary algebras, and repeats the
proof of the similar statement for ordinary derivations. 
\end{proof}

\begin{lemma}\label{lemma-sum}
Suppose $A$ is decomposed as the direct sum of ideals: 
$A = \bigoplus_{i=1}^n A_i$, and $\delta \ne 0$. Then
\[  
    \begin{multlined}
       \Der_\delta (A, M) \simeq  \Bigg\{ (D_1,\dots,D_n) \in \bigoplus_{i=1}^n \Der_\delta(A_i,M) ~\bigg| \\
        D_i(x) \bullet y + x \bullet D_j(y) = 0, \text{for any } x \in A_i, y \in A_j  ~i, j = 1, \dots, n,~  i \ne j \Bigg\} .
    \end{multlined}
\]

\end{lemma}

\begin{proof}
Again, a verbatim repetition of the proof of the Lie algebra case in 
\cite[Lemma~2]{delta-whitehead}.
\end{proof}

Now, assume first that $J$ is simple (and $M$ arbitrary). Since any 
finite-dimensional representation of $J$ is completely reducible, $M$ is 
decomposed as the direct sum of irreducible submodules: $M = \bigoplus_i M_i$.
If $\Der_\delta(J,M)$ does not vanish, then, according to Theorem~\ref{th-1}, 
either $\delta = 1$ and then $\Der_1(J,M)$ consists of inner derivations, or 
$\delta = \frac 12$ and the only nonzero summands in (\ref{eq-sum}) are those 
for which $M_i$ is isomorphic to the regular $J$-bimodule, in which cases 
$\Der_{\frac 12}(J,M_i) \simeq \Der_{\frac 12}(J,J)$ is the one-dimensional 
vector space linearly spanned by $\id_J$.

If $J$ is semisimple, then $J$ is decomposed into the direct sum of simple 
ideals, and Lemma \ref{lemma-sum} reduces the situation to simple cases.

To give a precise statement describing $\delta$-derivations of a semisimple 
Jordan algebra with values in a finite-dimensional module, by using combination
of Theorem \ref{th-1}, Lemma~\ref{lemma-5}, and Lemma~\ref{lemma-sum}, similar 
to Main Theorem in \cite{delta-whitehead} in the Lie algebra case, would be 
somewhat cumbersome. The main reason for this is the absence of the tensor 
product construction for bimodules over Jordan algebras. However, the scheme 
described above allows to settle the question effectively in each concrete case.

\section{\texorpdfstring{$\delta$-derivations, octonionic matrix algebras}{delta-derivations, octonionic matrix }}\label{sec-octonion}

Theorem \ref{th-1} (or, more exactly, reasoning in the preceding section based 
on it), allows to give a somewhat streamlined proof of the result established in
\cite{octonion}: triviality of $\delta$-derivations of the algebra 
$S^+(M_n(O),\mathbb J)$ of Hermitian $n \times n$ matrices over octonions.
Here $O$ is the algebra of octonions, $\mathbb J$ is the composition of the
matrix transposition and the standard involution on $O$, and the multiplication
in $S^+(M_n(O),\mathbb J)$ is defined according to the ``Jordan'' rule: 
\begin{equation}\label{eq-AB}
A \circ B = \frac 12 (AB + BA) .
\end{equation}

The streamlined proof we present here is based on the following general simple 
observation which can be of independent interest as a tool to derive triviality
of $\delta$-derivations of algebras from those of their subalgebras. Let us call a
(not necessarily Jordan) algebra $A$ \emph{$\delta$-challenged} if it satisfies
the conclusion of Theorem~\ref{th-1}, i.e., if for any finite-dimensional 
irreducible bimodule over $A$, $\Der_\delta(A,M) \ne 0$ implies that either 
$\delta=1$, or $\delta = \frac 12$, $M$ is the regular bimodule, and 
$\Der_{\frac 12}(A,M) = \Der_{\frac 12}(A,A)$ is linearly spanned by the 
identity map on $A$. 

\begin{lemma}\label{lemma-chall}
Let $A$ be a commutative algebra with unit, $S$ a simple unital 
$\delta$-challenged subalgebra of $A$ such that $A$ is decomposed into the 
direct sum of irreducible components as an $S$-bimodule, and among those 
components the only one isomorphic to the regular $S$-bimodule is $S$ itself. If
$D$ is a nonzero $\delta$-derivation of $A$ with values in the regular bimodule,
then either $\delta = 1$, or $\delta = \frac 12$ and $D$ is a scalar multiple of
the identity map.
\end{lemma}

\begin{proof}
Let $\delta \ne 1$. Restriction of $D$ to $S$ is a $\delta$-derivation of $S$ 
with values in $A$, and Lemma~\ref{lemma-5} together with the assumption that 
$S$ is $\delta$-challenged, implies that $\delta = \frac 12$ and 
$D(x) = \lambda x$ for some $\lambda \in K$ and any $x \in S$; in particular,
$D(1) = \lambda 1$. But according to Lemma~\ref{lemma-1}, $D(x) = x D(1)$ for 
any $x\in A$, whence $D(x) = \lambda x$ for any $x \in A$.
\end{proof}

\begin{theorem}[{\!\!\cite[Theorem 8]{octonion}}]
Let $D$ be a nonzero $\delta$-derivation of $S^+(M_n(O),\mathbb J)$
with values in the regular bimodule. Then either $\delta=1$, or 
$\delta=\frac 12$ and $D$ is a scalar multiple of the identity map.
\end{theorem}

In the proof below we will use the following shorthand notation:
\[M_n^+(K) = S^+(M_n(K),{}^\top),\] the Jordan algebra of $n \times n$ symmetric 
matrices; $M_n^-(K) = S^-(M_n(K),{}^\top)$, the vector space of $n \times n$ 
skew-symmetric matrices; and $O^- = S^-(O,{}^-)$, the $7$-dimensional vector 
space of octonions skew-symmetric with respect to the standard conjugation 
${}^-$ in $O$ (of course, the latter two ``minus'' vector spaces are a Lie 
algebra and a Malcev algebra respectively, but, in our entirely commutative 
``Jordan'' setting here, we do not need these nice and important facts).

\begin{proof}
As noted in \cite{octonion}, the algebra $S^+(M_n(O),\mathbb J)$ can be 
represented as the vector space direct sum 
$M_n^+(K) \oplus (M_n^-(K) \otimes O^-)$, where $M_n^+(K)$ is a (Jordan) 
subalgebra, the multiplication between $M_n^+(K)$ and $M_n^-(K) \otimes O^-$ is
performed by the action of $M_n^+(K)$ on the first tensor factor $M_n^-(K)$ via
the formula (\ref{eq-AB}), leaving the second tensor factor $O^-$ intact. 
($M_n^-(K) \otimes O^-$ is not a subalgebra, but the exact nice formula for 
multiplication of elements in this subspace will not concern us here; the 
interested reader can consult \cite{octonion} for details). Therefore, as an 
$M_n^+(K)$-bimodule, the whole algebra $S^+(M_n(O),\mathbb J)$ is decomposed 
into the direct sum of $8$ irreducible components: $M_n^+(K)$, the regular 
bimodule, and $7$ copies of the bimodule $M_n^-(K)$ (parametrized by $O^-$). By
Theorem \ref{th-1}, the Jordan algebra $M_n^+(K)$ is $\delta$-challenged, and 
thus Lemma~\ref{lemma-chall} is applicable. 
\end{proof}

The same idea -- restriction of a $\delta$-derivation to an appropriate 
subalgebra -- can be used to present a somewhat alternative proof of 
Theorem~\ref{th-1}. Namely, by a judicial choice of subalgebras of a simple (special) Jordan algebra
(see \cite{tval}), one can employ an induction by the dimension of an algebra, 
reducing the proof to the case of simple Jordan algebras not having a proper 
simple subalgebra. Over an algebraically closed field, such algebras are 
exhausted by the $1$-dimensional algebra, and the $3$-dimensional algebras $J(V,f)$ (where $V$ is 
$2$-dimensional) and $M_2^+(K)$. For the latter two algebras, the proof may go
the same way as in \S \ref{sec-delta}.

\subsection*{Acknowledgement}
Thanks are due to anonymous referees for remarks which led to improvement of the
text.

\EditInfo{May 16, 2024}{July 16, 2024}{Ivan Kaygorodov}

\EditInfo{May 16, 2024}{July 16, 2024}{Ivan Kaygorodov}

\end{document}